\documentclass{article}

\usepackage[english]{babel}

\usepackage[letterpaper,top=2cm,bottom=2cm,left=3cm,right=3cm,marginparwidth=1.75cm]{geometry}

\usepackage{amsmath}
\usepackage{amsthm}
\newtheorem{theorem}{Theorem}[section]
\newtheorem{lemma}{Lemma}[section]
\usepackage{amssymb}
\usepackage{graphicx}
\usepackage[colorlinks=true, allcolors=blue]{hyperref}
\usepackage{authblk}

\title{Higher-Order Congruence for Reciprocal Power Sums and Generalized Lehmer-Type Products}
\author{Zhenming Tang \and Hao Zhong}
\date{}
\begin{document}
\maketitle
 
\begin{abstract}
This paper investigates high-order congruences of reciprocal power sums and Lehmer-type products. Let $n\geq 1$ with $(n,6)=1$ and $e\in\{2,3,4,6\}$. For the reciprocal square sums
\begin{equation*}
    S(n)=\sum_{\substack{r=1 \\ (r,n)=1}}^{\lfloor n/e \rfloor}\frac{1}{r^2}
\end{equation*}
we already know the form of the congruence modulo $n$. In this paper, motivated by the known congruences, we first extend these results to certain reciprocal sums of odd order and establish a uniform congruence modulo $n$ for
\begin{equation*}
    S_m(n)=\sum_{\substack{r=1 \\ (r,n)=1}}^{\lfloor n/e \rfloor}\frac{1}{r^m}
\end{equation*}
We then study the generalized Lehmer-type product
\begin{equation*}
     \prod_{d \mid n}\binom{kd-1}{\lfloor d/e \rfloor}^{\mu(n/d)}
\end{equation*}
Although congruences modulo $n^3$ for this product have previously been obtained, higher-order
congruences do not admit a comparably simple closed form. To address this difficulty, we derive
an explicit truncated expansion in terms of complete exponential Bell polynomials. The results
provide a unified framework for explicit computation and algorithmic verification of higher-order
congruences involving reciprocal sums and related product expressions.
\end{abstract}

\section{Introduction}

In 1895, Morley\cite{2} proved the following beautiful and profound congruence involving binomial coefficients: for any prime $p>5$,
\begin{equation}
    (-1)^{\frac{p-1}{2}}\binom{p-1}{\frac{p-1}{2}}\equiv4^{p-1} \pmod {p^3}
    \label{eq:1}
\end{equation} 
Morley’s congruence and its subsequent extensions to composite moduli show that binomial coefficients, Lehmer-type products, reciprocal sums, and Bernoulli numbers are closely connected in
arithmetic modulo high powers of integers. Along this line of research, Cai and his collaborators established a series of congruences for combinatorial products and reciprocal sums parametrized by $\lfloor d/e \rfloor$\cite{3}\cite{4}.
In 2002, Cai\cite{3} extended Morley's congruence to integer moduli through a generalization of Lehmer's congruence:
\begin{equation}
    \prod_{d|n}\binom{d-1}{\frac{d-1}{2}}^{\mu(n/d)}\equiv (-1)^{\varphi(n)/2}4^{\varphi(n)}\begin{cases}
        \pmod {n^3} \ \text{if} \ 3 \nmid n \\
        \pmod {n^3/3} \ \text{if} \ 3\mid n
    \end{cases}
    \label{eq:2}
\end{equation} 
for odd $n>1$. When $n$ is an odd prime $p>5$, congruence \eqref{eq:2} reduces to Morley's congruence \eqref{eq:1}. Later in 2007, Cai et al.\cite{4} obtained analogous congruences in which $(d-1)/2$ is replaced by $\lfloor d/3 \rfloor$, $\lfloor d/4 \rfloor$ and $\lfloor d/6 \rfloor$ respectively, where $\lfloor x \rfloor$ denotes the largest integer not greater than $x$. Zhong et al.\cite{1} further derived quadratic reciprocal-sum congruences and product-type congruences associated with the Euler quotient, thereby laying a foundation for the study of higher-order moduli and more general weight exponents. As a continuation of the work of Cai\cite{4}, Zhong et al.\cite{1} proposed two main congruences below:
\begin{theorem}[\cite{1}, Theorem 1.1]
    Let $n$ be a positive integer and $(n,6)=1$. For $e=2,3,4,6$, we have
    \begin{equation}
        \sum_{\substack{r=1 \\ (r,n)=1}}^{\lfloor n/e \rfloor}\frac{1}{r^2}\equiv -J_e(n)n^{\varphi(n)-2}\varphi_{J_e}^{(2-\varphi(n))}(n)\frac{B_{\varphi(n)-1}(\frac1e)}{\varphi(n)-1} \pmod n
        \label{eq:3}
    \end{equation} 
    where $B_n$ is the $n$-th Bernoulli number, $B_n(x)$ is the $n$-th Bernoulli polynomial, $J_e(n)$ denotes the Jacobi symbol $(\frac ne)$ and $\varphi_f^{(k)}(n)$ is the generalized Euler totient, defined as 
    \begin{equation*}
        \varphi_f^{(k)}(n):=\sum_{d\mid n}(\frac nd)^kf(d)\mu(d)
    \end{equation*}
    for $f$ a number theoretic function and $k$ an integer.
\end{theorem}
If $f\equiv1$, then $\varphi_f^{(k)}(n)$ equals to Jordan totient function and it is easy to prove
\begin{equation*}
    \varphi_f^{(k)}(n)=n^k\prod_{p|n}(1-f(p)p^{-k})
\end{equation*}
if $f$ is multiplicative. And it is also easy to represent $J_e(n)$ by
\begin{equation*}
    J_e(n)=(\frac ne)=\begin{cases}
        1 \ \text{if}\ n\equiv 1 \pmod e \\
        -1 \ \text{if} \ n\equiv -1 \pmod e
    \end{cases}
\end{equation*}
since $(n,6)=1$.
\begin{theorem}[\cite{1}, Theorem 1.2]
    For any positive integer $k$ and odd $n>1$, it follows 
    \begin{equation}
        \prod_{d\mid n}\binom{kd-1}{(d-1)/2}^{\mu(n/d)}\equiv (-1)^{\varphi(n)/2}4^{k\varphi(n)}\begin{cases}
            \pmod {n^3} \ \text{if} \ 3\nmid n \\
            \pmod {n^3/3}\ \text{if} \ 3\mid n 
        \end{cases}
        \label{eq:4}
    \end{equation}
\end{theorem}

During the past decade, research in this direction has moved from refinements of individual classical congruences toward the development of unified methodological frameworks. Gy obtained $p$-adic expansions for generalized harmonic numbers\cite{7}; Chern extended congruences for Bernoulli-type harmonic sums to the setting of linear constraints over composite moduli\cite{8}; Akiyama and Kaneko derived new Lehmer-type divisibility relations from derivatives of cyclotomic polynomials\cite{9}; Ma and Li obtained new congruences involving multiple harmonic sums and Bernoulli polynomials evaluated at rational points\cite{10}; and Liu\cite{11}, Patchkoria\cite{12}, Komatsu–Liu\cite{13}, and Kellner\cite{14} respectively demonstrated the unifying role of Bernoulli and Bell structures in Apéry-like supercongruences, families of Bernoulli congruences, Lehmer–Euler numbers, and higher-order modular expansions of Wilson and Fermat quotients. These developments suggest that the central task in current research is no longer merely to increase the modulus order, but rather to construct a structured proof framework that is transferable, recursive, computationally effective, and amenable to verification.
The aim of this paper is to extend the reciprocal-sum congruence in Theorem 1 to certain odd powers $\frac{1}{r^m}$ and to lift the product congruence in Theorem 2 to a higher-order expansion expressed in terms of Bell polynomials.

These developments suggest that the central task is no longer merely to increase the modulus order of known congruences. Rather, it is important to construct a structured proof framework that is transferable, recursive, computationally effective, and amenable to verification. The aim of the present paper is twofold: first, to extend the reciprocal-square-sum congruence in Theorem 1.1 to certain reciprocal sums of odd order; and second, to lift the product congruence in Theorem 1.2 to a higher-order expansion expressed in terms of complete exponential Bell polynomials.

More precisely, this paper investigates the following two related questions. First, for composite moduli, can existing quadratic reciprocal sum congruences be extended to reciprocal sums of odd order and expressed uniformly in terms of Bernoulli polynomials? Second, can generalized Lehmer-type products be expanded modulo $n^{K+1}$, with all coefficients represented uniformly by Bell polynomials? To answer these questions, we first establish local congruences modulo prime powers and then derive the corresponding composite-modulus results by the Chinese remainder theorem. We subsequently use Möbius inversion, Newton’s identities for symmetric polynomials, and complete exponential Bell polynomials to obtain higher-order product expansions.

The main contributions of this paper are therefore as follows. First, we extend reciprocal square sum congruences to certain reciprocal sums of odd order. Second, we give a higher-order Bell polynomial representation for generalized Lehmer-type products and illustrate its computability through explicit low-order cases.

The remainder of the paper is organized as follows. Section 2 establishes the required prime-
power local congruences and then extends them to composite moduli. The main reciprocal sum
result is the following theorem.
\begin{theorem}
Let $n>1$ be a positive integer with $(n,6)=1$, $e\in\{2,3,4,6\}$.
Suppose that $m$ is an odd integer satisfying
\begin{equation*}
    3 \leq m \leq  \min_{p^l\mid\mid n}(\varphi(p^l)-l)
\end{equation*}
and
\begin{equation*}
    m \not\equiv 1 \pmod {p-1}
\end{equation*}
for every prime divisor $p$ of $n$, then
\begin{equation}
     \sum_{\substack{r=1 \\ (r,n)=1}}^{\lfloor n/e \rfloor}\frac{1}{r^m} \equiv n^{\varphi(n)-m}\varphi_{1}^{\left(m-\varphi(n)\right)}(n)\left(\frac{B_{\varphi(n)-m+1}(\frac 1e)}{\varphi(n)-m+1}-\frac{B_{\varphi(n)-m+1}}{\varphi(n)-m+1}\right) \pmod n
     \label{eq:5}
\end{equation}
\end{theorem}
Section 3 derives the higher-order product expansions and gives explicit low-order consequences. In particular, we prove the following result.
\begin{theorem}
Let $k,e$ be positive integers and let $n,K$ be positive integers satisfying $(n,K!)=1$. Then
    \begin{equation}
        \prod_{d \mid n}\binom{kd-1}{\lfloor d/e \rfloor}^{\mu(n/d)}\equiv (-1)^{\varphi_e(n)}\sum_{m=0}^K\frac{\bar{B}_m\left(-kS_1(n),\cdots,-(m-1)!k^mS_m(n)\right)}{m!}n^m \pmod {n^{K+1}}
        \label{eq:6}
    \end{equation}
where $\bar{B}_m$ denotes the $m$-th complete exponential Bell polynomial. For $K=2$, this congruence recovers the form obtained in \cite{1}.
\end{theorem}

\section{A Generalization of the Reciprocal-Sum Congruence}

\subsection{A Prime-Power Bernoulli Polynomial Congruence}

Cai et al. \cite{4} proved the following congruence. Let $p\geq 5$ be a prime, let $k\geq 2$ and let $l,t$ be positive with $p\nmid t$. If $s$ denotes the least positive residue of $p^l$ modulo $t$, then 
\begin{equation*}
    \sum_{r=1}^{\lfloor p^l/t \rfloor}(p^l-tr)^{2k}\equiv \frac{t^{2k}}{2k+1}\left(\frac{2k+1}{t}p^lB_{2k}-B_{2k+1}(\frac st)\right) \pmod {p^{3l-1}}
\end{equation*}
See also Lemma 2.1 of \cite{1}. We shall use the following refinement of this result.
Suppose $n=p^l$, where $p$ is an odd prime, $t$ is an integer and $p \nmid t$.
Define 
\begin{equation}
    S_m(p^l):=\sum_{\substack{r=1 \\ p \nmid r}}^{\lfloor p^l/t \rfloor}\frac{1}{r^m}
    \label{eq:7}
\end{equation}
By Euler's theorem,
\begin{equation*}
    r^{-1}\equiv r^{\varphi(p^l)-1} \pmod {p^l}
\end{equation*}
and hence 
\begin{equation}
    \frac{1}{r^m}\equiv r^{\varphi(p^l)-m} \pmod {p^l}
    \label{eq:8}
\end{equation}
Therefore,
\begin{equation}
    S_m(p^l)\equiv\sum_{\substack{r=1 \\ p \nmid r}}^{\lfloor p^l/t \rfloor}r^{\varphi(p^l)-m} \pmod {p^l}
    \label{eq:9}
\end{equation}

Since $p^l-tr\equiv -tr \pmod {p^l}$ and $(tr,p^l)=1$, we have
\begin{equation*}
    \frac 1r\equiv-\frac{t}{p^l-tr} \pmod {p^l}
\end{equation*}
and therefore
\begin{equation}
    \frac{1}{r^m}\equiv (-1)^mt^m\frac{1}{(p^l-tr)^m} \pmod {p^l}
    \label{eq:10}
\end{equation}
Moreover, since $(p^l-tr,p^l)=1$, another application of Euler's theorem gives $(p^l-tr)^{\varphi(p^l)}\equiv 1 \pmod {p^l}$, then
\begin{equation}
    \frac{1}{(p^l-tr)^m}\equiv (p^l-tr)^{h\varphi(p^l)-m} \pmod {p^l}
    \label{eq:11}
\end{equation}
where $h\in\mathbb{Z}$ is chosen so that $h\varphi(p^l)-m\geq 0$.
Combining \eqref{eq:10} and \eqref{eq:11}, we obtain
\begin{equation}
    \frac{1}{r^m}\equiv (-1)^mt^m(p^l-tr)^M \pmod {p^l}
    \label{eq:12}
\end{equation}
where
\begin{equation*}
    M=h\varphi(p^l)-m\geq 0
\end{equation*}
Thus,
\begin{align}
    S_m(p^l)&\equiv (-1)^mt^m\sum_{\substack{r=1 \\ p \nmid r}}^{\lfloor p^l/t \rfloor}(p^l-tr)^M \\
    &\equiv (-1)^mt^m\sum_{r=1}^{\lfloor p^l/t \rfloor}(p^l-tr)^M \pmod {p^l} \label{eq:14}
\end{align}
if $M\geq l$.

Let $N=\lfloor p^l/t \rfloor$, then $N=\frac{p^l-s}{t}$ since $p \nmid t$. Hence
\begin{equation}
    \sum_{r=1}^N(p^l-tr)^M=\sum_{j=0}^M\binom{M}{j}(p^l)^{M-j}(-t)^j\sum_{r=1}^Nr^j \label{eq:15}
\end{equation}
By Faulhaber's formula, we have 
\begin{equation}
    \sum_{r=1}^Nr^j=\frac{B_{j+1}(N+1)-B_{j+1}}{j+1} \label{eq:16}
\end{equation}
Observe that $N+1=\frac{p^l+t-s}{t}$, put $\delta=\frac{t-s}{t}\in[0,1)$ then $N+1=\frac{p^l}{t}+\delta$. By the additive formula of the Bernoulli polynomial 
\begin{equation*}
    B_{j+1}\left(\frac{p^l}{t}+\delta\right)=\sum_{i=0}^{j+1}\binom{j+1}{i}B_{j+1-i}(\delta)\left(\frac{p^l}{t}\right)^i
\end{equation*}
Together with \eqref{eq:16}, this gives
\begin{equation}
    \sum_{r=1}^Nr^j=\frac{1}{j+1}\left(\sum_{i=1}^{j+1}\binom{j+1}{i}B_{j+1-i}(\delta)\frac{p^{li}}{t^i}+\left(B_{j+1}(\delta)-B_{j+1}\right)\right) \label{eq:17}
\end{equation}
Substituting this expression into \eqref{eq:15}, we obtain
\begin{equation}
    \sum_{r=1}^N(p^l-tr)^M=\sum_{j=0}^M\binom{M}{j}(p^l)^{M-j}(-t)^j\frac{1}{j+1}\left(\sum_{i=1}^{j+1}\binom{j+1}{i}B_{j+1-i}(\delta)\frac{p^{li}}{t^i}+\left(B_{j+1}(\delta)-B_{j+1}\right)\right) \label{eq:18}
\end{equation}
Since terms with $M-j+i\geq 3$ vanish modulo $p^{3l-1}$, it suffices to consider the cases $M-j+i\leq 2$ only. This yields the following three contributions: Suppose $T_1$, $T_2$ and $T_3$ denote the residue of the right-hand side of \eqref{eq:18} modulo $p^{3l-1}$ when $j=M, i=1,2$; $j=M_1, i=1$ and $j=M-2, i=0$. Then
    \begin{equation}
        T_1=(-t)^M\left(\frac{B_{M+1}(\delta)-B_{M+1}}{M+1}+B_M(\delta)\frac{p^l}{t}+\frac{M}{2}B_{M-1}(\delta)\frac{p^{2l}}{t^2}\right)
        \label{eq:19}
    \end{equation}
    \begin{equation}
        T_2=(-t)^{M-1}\left(-B_Mp^l+B_M(\delta)p^l+MB_{M-1}(\delta)\frac{p^{2l}}{t}\right) \label{eq:20}
    \end{equation}
and
    \begin{equation}
        T_3=(-t)^{M-2}\frac{M}{2}p^{2l}\left(B_{M-1}(\delta)-B_{M-1}\right)
        \label{eq:21}
    \end{equation}
Consequently,
\begin{equation}
    \sum_{r=1}^N(p^l-tr)^M\equiv T_1+T_2+T_3 \pmod {p^{3l-1}} \label{eq:22}
\end{equation}
If $M$ is even, this simplifies to
\begin{equation}
     \sum_{r=1}^N(p^l-tr)^M\equiv t^{M-1}p^lB_M-\frac{t^M}{M+1}B_{M+1}\left(\frac st\right) \pmod {p^{3l-1}}
     \label{eq:23}
\end{equation}
This is the congruence originally obtained in \cite{4}. 
If $M$ is odd, an analogous calculation based on \eqref{eq:22} gives
\begin{equation}
    \sum_{r=1}^N(p^l-tr)^M\equiv\frac{M}{2}B_{M-1}t^{M-2}p^{2l}+\frac{t^M}{M+1}\left(B_{M+1}-B_{M+1}\left(\frac st\right)\right) \pmod {p^{3l-1}}
    \label{eq:24}
\end{equation}
We have therefore obtained the following generalized congruence.
\begin{lemma}
    If $p\geq 5$ is a prime and $M\geq 3$, $l,t$ are positive integers where $p\nmid t$ and $s$ is the least positive residue of $p^l$ modulo $t$, then 
    \begin{equation*}
        \sum_{r=1}^N(p^l-tr)^M\equiv t^{M-1}p^lB_M-\frac{t^M}{M+1}B_{M+1}\left(\frac st\right) \pmod {p^{3l-1}}
    \end{equation*}
    if $M$ is even, while
    \begin{equation*}
        \sum_{r=1}^N(p^l-tr)^M\equiv\frac{M}{2}B_{M-1}t^{M-2}p^{2l}+\frac{t^M}{M+1}\left(B_{M+1}-B_{M+1}\left(\frac st\right)\right) \pmod {p^{3l-1}}
    \end{equation*}
    if $M$ is odd.
\end{lemma}
Reducing Lemma 2.1 modulo the required power of $p$, we immediately obtain the following consequence.
\begin{theorem}
    If $p\geq 5$ is a prime and $M\geq 3$, $l,t$ are positive integers where $p\nmid t$ and $s$ is the least positive residue of $p^l$ modulo $t$, then
    \begin{equation}
        \sum_{r=1}^N(p^l-tr)^M\equiv -\frac{t^M}{M+1}B_{M+1}\left(\frac st\right) \pmod {p^l}
        \label{eq:25}
    \end{equation}
    if $M$ is even, while
    \begin{equation}
        \sum_{r=1}^N(p^l-tr)^M\equiv \frac{t^M}{M+1}\left(B_{M+1}-B_{M+1}\left(\frac st\right)\right) \pmod {p^{2l}}
        \label{eq:26}
    \end{equation}
    if $M$ is odd, which is the formula (9) in \cite{4}.
\end{theorem}

We now assume that $M$ is odd and sufficiently large for the preceding reduction to be valid. Substituting \eqref{eq:26} into \eqref{eq:14}, we obtain
\begin{equation}
    S_m(p^l)\equiv (-1)^mt^{M+m}\left(\frac{B_{M+1}-B_{M+1}\left(\frac st\right)}{M+1}\right) \pmod {p^l}
    \label{eq:27}
\end{equation}
Moreover, suppose $m$ is odd and set
\begin{equation*}
    M=p^h\varphi(p^l)-m=\varphi(p^{l+h})-m
\end{equation*}
Then
\begin{align}
    S_m(p^l)&\equiv-t^{M+m}\frac{B_{M+1}-B_{M+1}\left(\frac st\right)}{M+1} \\
    &\equiv \frac{B_{M+1}\left(\frac st\right)-B_{M+1}}{M+1} \pmod {p^l}
    \label{eq:29}
\end{align}
because $M+m$ is even and $t^{\varphi(p^{l+h})}\equiv 1 \pmod {p^l}$.

We next recall a Kummer-type congruence for Bernoulli polynomials, which will be used to replace the index $M+1$ by a lower congruent index.
\begin{lemma}[See\cite{15}, Lemma 2.5]
    Let $p$ be an odd prime and $l\geq1$. Let $N$ and $N'$ be positive even integers such that $N,N'>l$,
    \begin{equation*}
        N \equiv N' \mod \varphi(p^l)
    \end{equation*}
    and $p-1\nmid N,N'$. Then, for every rational number $x$ whose denominator is not divisible by $p$, one has
    \begin{equation}
        \frac{B_N(x)}{N} \equiv \frac{B_{N'}(x)}{N'} \pmod {p^l}
        \label{eq:30}
    \end{equation}
\end{lemma}

Take $N=M+1,N'=\varphi(p^l)-m+1$ and $x=\frac st$. Since
\begin{equation*}
    M+1=\varphi(p^{l+h})-m+1=p^h\varphi(p^l)-m+1
\end{equation*}
we have 
\begin{equation*}
    M+1\equiv -m+1 \pmod {\varphi(p^l)}
\end{equation*}
Moreover, $\frac st$ is coprime with $p$ due to the definition of $s$. Applying \eqref{eq:30}, we obtain
\begin{equation}
    \frac{B_{M+1}\left(\frac st\right)}{M+1}-\frac{B_{M+1}}{M+1}\equiv \frac{B_{\varphi(p^l)-m+1}\left(\frac{s}{t}\right)}{\varphi(p^l)-m+1}-\frac{B_{M+1}}{M+1} \pmod {p^l}
    \label{eq:31}
\end{equation}

Now set $t=e$, Since $e\in \{2,3,4,6\}$ and $(p^l,6)=1$, the least positive residue $s$ of $p^l$ modulo $e$ is either $1$ or $e-1$. Thus 
\begin{equation*}
    J_e(p^l)=\begin{cases}
        1, \ p^l\equiv 1 \pmod e \\
        -1, \ p^l\equiv -1 \pmod e
    \end{cases}
\end{equation*}
Equivalently, $\frac se=\frac 1e$ when $J_e(p^l)=1$, and $\frac se=1-\frac 1e$ when $J_e(p^l)=-1$.
If $s=1$, one has
    \begin{equation*}
        S_m(p^l)\equiv J_e(p^l)\left(\frac{B_{\varphi(p^l)-m+1}\left(\frac 1e\right)}{\varphi(p^l)-m+1}-\frac{B_{M+1}}{M+1}\right) \pmod {p^l}
    \end{equation*}
and if $s=e-1$, note that 
    \begin{align*}
        B_{\varphi(p^l)-m+1}\left(1-\frac 1e\right)&=(-1)^{\varphi(p^l)-m+1}B_{\varphi(p^l)-m+1}\left(\frac 1e\right) \\
        &=B_{\varphi(p^l)-m+1}\left(\frac 1e\right)
    \end{align*}
and $J_e(p^l)=-1$, we obtain
    \begin{equation*}
        S_m(p^l)\equiv -J_e(p^l)\left(\frac{B_{\varphi(p^l)-m+1}\left(\frac 1e\right)}{\varphi(p^l)-m+1}-\frac{B_{M+1}}{M+1}\right) \pmod {p^l}
    \end{equation*}
We therefore obtain the following odd-order analogue of congruence (3.1) in \cite{1}.
\begin{theorem}
    Let $p\geq5$ be a prime, $m\geq 3$ be an odd integer satisfying $m\leq \varphi(p^l)-l$ and $m \not\equiv 1 \mod p-1$. For $e=2,3,4,6$, we have
    \begin{equation}
        S_m(p^l)\equiv \frac{B_{\varphi(p^l)-m+1}\left(\frac 1e\right)}{\varphi(p^l)-m+1}-\frac{B_{\varphi(p^{l+h})-m+1}}{\varphi(p^{l+h})-m+1} \pmod {p^l}
        \label{eq:32}
    \end{equation}
    where $h$ is an arbitrary integer such that $\varphi(p^{l+h})-m\geq0$. 
\end{theorem} 

\subsection{A Lifting Relation for Truncated Reciprocal Sums}
Following the method of \cite{1}, we next prove the following lifting relation:
\begin{equation}
    \sum_{\substack{r=1 \\ p \nmid r}}^{ \lfloor cp^l/e \rfloor}\frac{1}{r^m}\equiv J_e(c)^{m-1}S_m(p^l) \pmod {p^l}
    \label{eq:33}
\end{equation}
for any positive integer $c$ that $(c,e)=1$ and $p-1\nmid m$.

To this end, we first need the following higher-order version of Lemma 2.2 in \cite{1}.
\begin{lemma}
    Let $p \geq 5$ be a prime, let $l \geq 1$, and let $m$ be an integer satisfying $p-1 \nmid m$. Then
    \begin{equation}
        \sum_{\substack{i=1 \\ (i,n)=1}}^{n-1}\frac{1}{i^m}\equiv 0 \pmod {p^l}
        \label{eq:34}
    \end{equation}
\end{lemma}
\begin{proof}
By Euler's theorem, $i^{\varphi(n)}\equiv 1 \pmod n$. For fixed $m$, take $t\in\mathbb{Z}$ large enough so that
\begin{equation*}
    k=t\varphi(n)-m\geq 0
\end{equation*}
Then $i^{-m}\equiv i^k \pmod n$, and hence
\begin{equation}
    \sum_{\substack{i=1 \\ (i,p)=1}}^{n-1}\frac{1}{i^m} \equiv \sum_{\substack{i=1 \\ (i,p)=1}}^{n-1}i^k \pmod {p^l}
    \label{eq:35}
\end{equation}
Let \begin{equation*}
     T_k(n)=\sum_{\substack{i=1 \\ (i,p)=1}}^{n-1}i^k
\end{equation*}
Since $(\mathbb{Z}/n\mathbb{Z})^\times$ is a cyclic group, let $(\mathbb{Z}/n\mathbb{Z})^\times=:\langle g \rangle$, where $g$ is a primitive root modulo $p^l$. Hence 
\begin{equation}
    T_k(p^l)=\sum_{t=0}^{\varphi(p^l)-1}(g^t)^k=\sum_{t=0}^{\varphi(p^l)-1}(g^k)^t \label{eq:36}
\end{equation}
Put $\omega=g^k$. The order of $\omega$ in group $(\mathbb{Z}/n\mathbb{Z})^\times$ is
\begin{equation*}
    d:=\frac{\varphi(p^l)}{(\varphi(p^l),k)}
\end{equation*}
Thus
\begin{equation} 
    \omega^d\equiv 1 \pmod {p^l} \label{eq:37}
\end{equation}
and
\begin{equation}
    T_k(p^l)=\frac{\varphi(p^l)}{d}\sum_{t=0}^{d-1}\omega^t \label{eq:38}
\end{equation}
Note that $g$ is also a primitive root modulo $p$, the order of $\omega$ in group $(\mathbb{Z}/p\mathbb{Z})^\times$ is $\frac{p-1}{(p-1,k)}>1$ since $p-1 \nmid k$, therefore
\begin{equation*}
    \omega \not\equiv 1 \pmod p
\end{equation*}
Equivalently, $p \nmid \omega-1$ and $\omega-1$ is an invertible element in $(\mathbb{Z}/p\mathbb{Z})^\times$. Moreover,  $(\mathbb{Z}/p^l\mathbb{Z})^\times$. Combining \eqref{eq:37}, we obtain
\begin{equation}
    \sum_{t=0}^{d-1}\omega^t=\frac{\omega^d-1}{\omega-1}\equiv 0 \pmod {p^l} \label{eq:39}
\end{equation}
It follows that
\begin{equation}
    T_k(p^l)\equiv \sum_{\substack{i=1 \\ (i,n)=1}}^{n-1}\frac{1}{i^m} \equiv 0 \pmod {p^l} \label{eq:40}
\end{equation}
which proves the lemma.
\end{proof}

We now prove \eqref{eq:33}. If $c\equiv 1 \pmod e$, write $c=ek+1$ for some nonnegative integer $k$. Then
    \begin{align*}
        \sum_{\substack{r=1 \\ p \nmid r}}^{ \lfloor cp^l/e \rfloor}\frac{1}{r^m}&=\sum_{\substack{r=1 \\ p \nmid r}}^{kp^l}\frac{1}{r^m}+\sum_{\substack{r=kp^l+1 \\ p \nmid r}}^{kp^l+\lfloor p^l/e \rfloor}\frac{1}{r^m} \\
        &\equiv \sum_{a=0}^{k-1}\sum_{\substack{b=0 \\ p \nmid b}}^{p^l-1}\frac{1}{(ap^l+b)^m}+\sum_{\substack{r=1 \\ p \nmid r}}^{\lfloor p^l/e \rfloor}\frac{1}{(kp^l+r)^m} \\
        &\equiv \sum_{a=0}^{k-1}\sum_{\substack{b=0 \\ p \nmid b}}^{p^l-1}\frac{1}{b^m}+\sum_{\substack{r=1 \\ p \nmid r}}^{\lfloor p^l/e \rfloor}\frac{1}{r^m} \\
        &\equiv k\sum_{\substack{b=0 \\ p \nmid b}}^{p^l-1}\frac{1}{b^m}+\sum_{\substack{r=1 \\ p \nmid r}}^{ \lfloor p^l/e \rfloor}\frac{1}{r^m} \\
        &\equiv \sum_{\substack{r=1 \\ p \nmid r}}^{\lfloor p^l/e \rfloor}\frac{1}{r^m} \\
        &= S_m(p^l) \pmod {p^l}
    \end{align*}
And if $c\equiv -1 \pmod e$, write $c=ek-1$. Then
    \begin{align*}
        \sum_{\substack{r=1 \\ p \nmid r}}^{\lfloor cp^l/e \rfloor}\frac{1}{r^m} &= \sum_{\substack{r=1 \\ p \nmid r}}^{kp^l-\lfloor p^l/e\rfloor-1}\frac{1}{r^m} \\
        &\equiv -\sum_{\substack{r=kp^l-\lfloor p^l/e \rfloor \\ p \nmid r}}^{kp^l-1} \frac{1}{r^m} \\
        &\equiv -(-1)^m\sum_{\substack{r=1 \\ p \nmid r}}^{\lfloor p^l/e \rfloor} \frac{1}{r^m} \\
        &\equiv (-1)^{m-1}S_m(p^l) \\
        &\equiv J_e(c)^{m-1}S_m(p^l) \pmod {p^l}
    \end{align*}
So \eqref{eq:33} is valid. Furthermore, let $m$ be an odd integer. If $p^l||n$, taking $c=\frac{n}{p^l}$ in \eqref{eq:33} and combining the result with \eqref{eq:32}, we get
\begin{equation}
    \sum_{\substack{r=1 \\ p \nmid r}}^{\lfloor n/e \rfloor}\frac{1}{r^m} \equiv \frac{B_{\varphi(p^l)-m+1}\left(\frac 1e\right)}{\varphi(p^l)-m+1}-\frac{B_{\varphi(p^{l+h})-m+1}}{\varphi(p^{l+h})-m+1} \pmod {p^l}
    \label{eq:41}
\end{equation}
because $J_e\left(\frac{n}{p^l}\right)^{m-1}=1$.

\subsection{The Proof of Theorem 1.3}
Suppose $\varphi(n)-m\geq l$, we prove that 
\begin{equation}
    \sum_{\substack{r=1 \\ (r,n)=1}}^{\lfloor n/e \rfloor}\frac{1}{r^m} \equiv n^{\varphi(n)-m}\varphi_{1}^{(m-\varphi(n))}(n)\left(\frac{B_{\varphi(p^l)-m+1}\left(\frac 1e\right)}{\varphi(p^l)-m+1}-\frac{B_{\varphi(p^{l+h})-m+1}}{\varphi(p^{l+h})-m+1}\right) \pmod {p^l}
    \label{eq:42}
\end{equation}
Assume $p_1,p_2,\cdots,p_u$ are different prime factors of $n$. Since $\varphi(n)-1\geq\varphi(p^l)-1=p^{l-1}(p-1)-1\geq 4\cdot5^{l-1}-1>l$, we obtain
\begin{align*}
    \sum_{\substack{r=1 \\ (r,n)=1}}^{\lfloor n/e \rfloor}\frac{1}{r^m}&=\sum_{\substack{r=1 \\ p \nmid r}}^{\lfloor n/e \rfloor}\frac{1}{r^m}-\sum_i\sum_{\substack{r=1 \\ p \nmid r \\ p_i \mid r}}^{\lfloor n/e \rfloor}\frac{1}{r^m}+\sum_{i,j}\sum_{\substack{r=1 \\ p \nmid r \\ p_ip_j \mid r}}^{\lfloor n/e \rfloor}\frac{1}{r^m}+\cdots+(-1)^u\sum_{\substack{r=1 \\ p \nmid r \\ p_1p_2\cdots p_u \mid r}}^{\lfloor n/e \rfloor}\frac{1}{r^m} \\
    &\equiv \left(1-\sum_i\frac1{p_i^m}+\sum_{i,j}\frac{1}{p_i^mp_j^m}+\cdots+(-1)^u\frac{1}{p_1^mp_2^m\cdots p_u^m}\right)\left(\frac{B_{\varphi(p^l)-m+1}\left(\frac 1e\right)}{\varphi(p^l)-m+1}-\frac{B_{\varphi(p^{l+h})-m+1}}{\varphi(p^{l+h})-m+1}\right) \\
    &\equiv \prod_{\substack{q|\frac{n}{p^l}}}\left(1-\frac{1}{q^m}\right)\left(\frac{B_{\varphi(p^l)-m+1}\left(\frac 1e\right)}{\varphi(p^l)-m+1}-\frac{B_{\varphi(p^{l+h})-m+1}}{\varphi(p^{l+h})-m+1}\right) \\
    &\equiv \prod_{\substack{q\mid n \\ q \neq p}}\left(1-q^{\varphi(n)-m}\right) \left(\frac{B_{\varphi(p^l)-m+1}\left(\frac 1e\right)}{\varphi(p^l)-m+1}-\frac{B_{\varphi(p^{l+h})-m+1}}{\varphi(p^{l+h})-m+1}\right) \\
    &\equiv \prod_{q\mid n}(1-q^{\varphi(n)-m}) \left(\frac{B_{\varphi(p^l)-m+1}\left(\frac 1e\right)}{\varphi(p^l)-m+1}-\frac{B_{\varphi(p^{l+h})-m+1}}{\varphi(p^{l+h})-m+1}\right) \\
    &\equiv n^{\varphi(n)-m}\varphi_1^{(m-\varphi(n))}(n)\left(\frac{B_{\varphi(p^l)-m+1}\left(\frac 1e\right)}{\varphi(p^l)-m+1}-\frac{B_{\varphi(p^{l+h})-m+1}}{\varphi(p^{l+h})-m+1}\right) \pmod {p^l}
\end{align*}
since $1-p^{\varphi(n)-m}\equiv 1 \pmod {p^l}$. Hence \eqref{eq:42} follows.

It remains to replace the local Bernoulli-polynomial indices with the corresponding global index. We use Sun’s congruence for Bernoulli polynomials.
\begin{lemma}[\cite{6}]
    Let $a\in\mathbb{Z}, k,q,m\in\mathbb{Z}^+$ and $(m,q)=1$. Then
    \begin{equation}
        \frac 1k\left(m^kB_k\left(\frac{x+a}{m}\right)-B_k(x)\right)\equiv \sum_{j=0}^{q-1}\left(\lfloor\frac{a+jm}{q} \rfloor+\frac{1-m}{2}\right)\left(x+a+jm\right)^{k-1} \pmod q
        \label{eq:43}
    \end{equation}
\end{lemma}

Applying Lemma 2.4 with the local indices $k=\varphi(p^l)-m+1,m=e,x=0,a=1$ and $q=p^l$, observe that $e^{\varphi(p^l)}\equiv 1 \pmod {p^l}$ and $B_{k}(0)=B_k$, one obtains
\begin{align*}
    \frac{B_{\varphi(p^l)-m+1}\left(\frac 1e\right)}{\varphi(p^l)-m+1} &\equiv e\sum_{j=0}^{p^l-1}\left(\lfloor\frac{1+je}{p^l} \rfloor+\frac{1-e}{2}\right)(1+je)^{\varphi(p^l)-m}+\frac{B_{\varphi(p^l)-m+1}}{\varphi(p^l)-m+1} \\
    &\equiv e\sum_{\substack{j=0 \\ (p,1+je)=1}}^{p^l-1}\left(\lfloor\frac{1+je}{p^l} \rfloor+\frac{1-e}{2}\right)(1+je)^{-m}+\frac{B_{\varphi(p^l)-m+1}}{\varphi(p^l)-m+1} \pmod {p^l}
\end{align*}
Changing $k$ to $\varphi(n)-m+1$, one can similarly deduce that
\begin{equation}
    \frac{B_{\varphi(n)-m+1}(\frac 1e)}{\varphi(n)-m+1} \equiv e\sum_{\substack{j=0 \\ (p,1+je)=1}}^{p^l-1}\left(\lfloor\frac{1+je}{p^l} \rfloor+\frac{1-e}{2}\right)(1+je)^{-m}+\frac{B_{\varphi(p^l)-m+1}}{\varphi(p^l)-m+1} \pmod {p^l} \label{eq:44}
\end{equation}
Thus, 
\begin{equation}
    \frac{B_{\varphi(p^l)-m+1}\left(\frac 1e\right)}{\varphi(p^l)-m+1} \equiv \frac{B_{\varphi(n)-m+1}\left(\frac 1e\right)}{\varphi(n)-m+1} \pmod {p^l}
    \label{eq:45}
\end{equation} 
Assume that $m\leq \varphi(p^l)$, we may take $h=0$. In the same way, if we take $k=\varphi(p^l)-m+1,m=e,x=0,a=0$ and $q=p^l$ in \eqref{eq:43}, one can similarly obtain 
\begin{equation*}
    \frac{B_{\varphi(p^l)-m+1}}{\varphi(p^l)-m+1} \equiv \frac{B_{\varphi(n)-m+1}}{\varphi(n)-m+1} \mod p^l \tag{45'}\label{eq:45'}
\end{equation*} 
Since this congruence holds for every prime power $p^l||n$, combining \eqref{eq:45}, \eqref{eq:45'} and \eqref{eq:42}, the Chinese remainder theorem yields the desired congruence modulo $n$. This proves Theorem 1.3.

\section{Higher-Order Product Congruences and Bell Polynomial Expansions}
For a fixed positive integer $k$ and $e$, define 
\begin{equation}
    T_n:=\prod_{\substack{r=1 \\ (r,n)=1}}^{\lfloor n/e \rfloor} \frac{kn-r}{r}=(-1)^{\varphi_e(n)}\prod_{\substack{r=1 \\ (r,n)=1}}^{\lfloor n/e \rfloor}\left(1-\frac{kn}{r}\right) \label{eq:46}
\end{equation}
where
\begin{equation*}
    \varphi_e(n)=\#\{r|(r,n)=1,r=1,2,\cdots,\lfloor n/e \rfloor\}
\end{equation*}
Expanding the product through its logarithm in the truncated ring $\mathbb{Z}/n^{K+1}\mathbb{Z}$, we obtain
\begin{align*}
    \log \prod_{\substack{r=1 \\ (r,n)=1}}^{\lfloor n/e \rfloor}\left(1-\frac{kn}{r}\right) &= \sum_{\substack{r=1 \\ (r,n)=1}}^{\lfloor n/e \rfloor}\log\left(1-\frac{kn}{r}\right) \\
    &\equiv \sum_{\substack{r=1 \\ (r,n)=1}}^{\lfloor n/e \rfloor}\left(-\sum_{m=1}^K\frac{\left(\frac{kn}{r}\right)^m}{m}\right) \\
    &\equiv-\sum_{m=1}^K\frac{(kn)^m}{m} \sum_{\substack{r=1 \\ (r,n)=1}}^{\lfloor n/e \rfloor} \frac{1}{r^m} \\
    &\equiv -\sum_{m=1}^K\frac{(kn)^m}{m}S_m(n) \pmod {n^{K+1}}
\end{align*}
Exponentiating the truncated logarithmic expansion gives
\begin{align*}
    \prod_{\substack{r=1 \\ (r,n)=1}}^{\lfloor n/e \rfloor}\left(1-\frac{kn}{r}\right) &= \exp\left(\log\prod_{\substack{r=1 \\ (r,n)=1}}^{\lfloor n/e \rfloor}\left(1-\frac{kn}{r}\right) \right) \\
    &\equiv\exp{\left(-\sum_{m=1}^K\frac{(kn)^m}{m}S_m(n)\right)} \\
    &\equiv \sum_{j=0}^{\infty} \frac{1}{j!}\left(-\sum_{m=1}^K\frac{(kn)^m}{m}S_m(n)\right)^j \\
    &\equiv: \sum_{t=0}^K c_tn^t \pmod {n^{K+1}}
\end{align*}
where $c_0=1,c_1=-kS_1(n),c_2=\frac{k^2}{2}\left(S_1^2(n)-S_2(n)\right),\cdots$. 

Recall the definition of the complete exponential Bell Polynomial:
\begin{equation}
    \exp{\left(\sum_{j=1}^\infty a_j\frac{x^j}{j!}\right)}=:\sum_{m=0}^\infty \frac{\bar{B}_m(a_1,a_2,\cdots,a_m)}{m!}x^m \label{eq:47}
\end{equation}
where $\bar{B}_m(a_1,a_2,\cdots,a_m)$ denotes the $m$-{th} complete exponential Bell polynomial with coefficients $a_1,a_2,\cdots,a_m$, and $\bar{B}_0=1$. Taking $x=n$ and
\begin{equation*}
    a_j=-k^j(j-1)!S_j(n)
\end{equation*}
we obtain
\begin{equation}
    \exp{\left(-\sum_{m=1}^K\frac{(kn)^m}{m}S_m(n)\right)}\equiv \sum_{m=0}^K\frac{\bar{B}_m\left(-kS_1(n),\cdots,-(m-1)!k^mS_m(n)\right)}{m!}n^m \pmod {n^{K+1}} \label{eq:48}
\end{equation}
Therefore
\begin{equation}
    T_n\equiv (-1)^{\varphi_e(n)}\sum_{m=0}^K\frac{\bar{B}_m\left(-kS_1(n),-k^2S_2(n),\cdots,-(m-1)!k^mS_m(n)\right)}{m!}n^m \pmod {n^{K+1}} \label{eq:49}
\end{equation}

Let 
\begin{equation*}
    G_e(n)=\binom{kn-1}{\lfloor n/e \rfloor} =\prod_{r=1}^{\lfloor n/e\rfloor}\frac{kn-r}{r}
\end{equation*}
For each $r$, write $r=du$ where $d=(r,n)$. Then $u\leq \frac{n}{de}$ and $(u,\frac{n}{d})=1$. Hence
\begin{equation}
    G_e(n) =\prod_{d\mid n}\prod_{\substack{1\le u\le \lfloor \frac{n}{de} \rfloor \\(u,\frac nd)=1}} \frac{\frac {kn}{d}-u}{u} =\prod_{d\mid n} T_{n/d} =\prod_{d\mid n}T_d \label{eq:50}
\end{equation}
By the multiplicative form of Möbius inversion, it follows that
\begin{equation}
    T_n=\prod_{d \mid n}\binom{kd-1}{\lfloor d/e \rfloor}^{\mu(n/d)} \label{eq:51}
\end{equation}
Then we obtain Theorem 1.4.

Taking $K=2$, we obtain
\begin{align*}
    \prod_{d \mid n}\binom{kd-1}{\lfloor d/e \rfloor}^{\mu(n/d)} &\equiv (-1)^{\varphi_e(n)}\left(1-knS_1(n)+\frac{n^2}{2}\left(k^2S_1^2(n)-k^2S_2(n)\right)\right) \\
    &\equiv (-1)^{\varphi_e(n)}\left(1-knS_1(n)+\frac{k^2n^2}{2}\left(S_1^2(n)-S_2(n)\right)\right) \pmod {n^3}
\end{align*}
which agrees with the form obtained in \cite{1}. Furthermore, taking $K=3$, we similarly obtain 
\begin{equation}
\begin{split}
    \prod_{d \mid n}\binom{kd-1}{\lfloor d/e \rfloor}^{\mu(n/d)} &\equiv (-1)^{\varphi_e(n)}\bigg(1-knS_1(n)+\frac{k^2n^2}{2}\left(S_1^2(n)-S_2(n)\right)\\
    &-\frac{k^3n^3}{6}\left(S_1^3(n)-3S_1(n)S_2(n)+2S_3(n)\right)\bigg) \pmod {n^4}
    \label{eq:52}
\end{split}
\end{equation}

For computational purposes, we also record a recursive form of Theorem 1.4.
Define $C_0=1$ and for $t\geq 1$, set 
\begin{equation}
    C_t=-\frac1t\sum_{j=1}^{t} k^j S_j(n) C_{t-j} \label{eq:53}
\end{equation}
Then
\begin{equation}
    \prod_{d\mid n}\binom{kd-1}{\lfloor d/e \rfloor}^{\mu(n/d)} \equiv (-1)^{\varphi_e(n)}\sum_{t=0}^{K} C_t n^t \pmod{n^{K+1}}
    \label{eq:54}
\end{equation}
In particular, we have
\begin{equation*}
    C_1=-kS_1(n), \quad C_2=\frac{k^2}{2}\left(S_1(n)^2-S_2(n)\right)
\end{equation*}
and
\begin{equation*}
    C_3=-\frac{k^3}{6}\left(S_1(n)^3-3S_1(n)S_2(n)+2S_3(n)\right)
\end{equation*}

\bibliographystyle{plain}
\bibliography{sample}

\end{document}